\documentclass[final]{amsart}
\usepackage{amssymb,amsmath, color}

\usepackage{amssymb}

\newtheorem{Theorem}{Theorem}[section]
\newtheorem{lemma}{Lemma}[section]

\theoremstyle{remark}
\newtheorem{remark}{Remark}[section]

\newcommand{\R}{{\mathbb R}}

\newcommand{\argmin}{{\rm argmin}\kern 0.12em}

\usepackage{ulem}

\begin{document}

\title{COMBINING FAST INERTIAL DYNAMICS  FOR CONVEX OPTIMIZATION  WITH TIKHONOV REGULARIZATION.}

\author{Hedy Attouch}

\address{Institut Montpelli\'erain Alexander Grothendieck, UMR CNRS 5149, Universit\'e Montpellier 2, place Eug\`ene Bataillon,
34095 Montpellier cedex 5, France}
\email{hedy.attouch@univ-montp2.fr}

\author{Zaki Chbani}
\address{Cadi Ayyad university, Faculty of Sciences Semlalia, Mathematics, 40000 Marrakech, Morroco}
\email{chbaniz@uca.ma}

\date{February 5, 2016}


\keywords{Convex optimization; fast gradient methods;
FISTA algorithm; hierarchical minimization; inertial dynamics; Lyapunov analysis;  Nesterov accelerated method; Tikhonov approximation; vanishing viscosity}

\begin{abstract}
In a Hilbert space setting $\mathcal H$,  we study the  convergence properties as $t \to + \infty$ of the trajectories of  the second-order differential equation
\begin{equation*}
 \mbox{(AVD)}_{\alpha, \epsilon} \quad \quad \ddot{x}(t) + \frac{\alpha}{t} \dot{x}(t) + \nabla \Phi (x(t)) + \epsilon (t) x(t) =0,
\end{equation*}
where $\nabla\Phi$ is the gradient of a convex continuously differentiable function $\Phi: \mathcal H \to  \mathbb R$, $\alpha$ is a positive parameter, and $\epsilon (t) x(t)$ is a
Tikhonov regularization term, with $\lim_{t \to \infty}\epsilon (t) =0$. 
In this damped inertial system, the  damping coefficient $\frac{\alpha}{t}$ vanishes asymptotically, but not too quickly, a key property  to obtain
rapid convergence of the values. In the case $\epsilon (\cdot) \equiv 0$, 
this dynamic has been highlighted recently by Su, Boyd, and Cand\`es as a continuous version of the Nesterov accelerated method.
Depending on the speed of convergence of 
$\epsilon (t)$ to zero, we analyze the convergence properties of the trajectories of $\mbox{(AVD)}_{\alpha, \epsilon}$. We obtain results ranging from the rapid convergence of $\Phi (x(t))$ to $\min \Phi$ when $\epsilon (t)$ decreases rapidly to zero, up to the strong ergodic convergence of the trajectories to the element of minimal norm of the set of minimizers of $\Phi$, when $\epsilon (t)$ tends slowly to zero.

\end{abstract}

\thanks{With the support of  ECOS  grant C13E03,
Effort sponsored by the Air Force Office of Scientific Research, Air Force Material Command, USAF, under grant number FA9550-14-1-0056.}
\maketitle

\vspace{0.3cm}

\section{Introduction}
Throughout the paper, $\mathcal H$ is a real Hilbert space which is endowed with the scalar product $\langle \cdot,\cdot\rangle$, with $\|x\|^2= \langle x,x\rangle    $ for  $x\in \mathcal H$.
Let $\Phi : \mathcal H \rightarrow \mathbb R$ be a convex differentiable function,  whose gradient   $\nabla \Phi$ is Lipschitz continuous on  bounded sets.
We aim at solving  by rapid methods the convex minimization problem
\begin{equation}\label{edo0001}
 \min \left\lbrace  \Phi (x) : \ x \in \mathcal H \right\rbrace ,
\end{equation}
whose solution set $S=\argmin \Phi$ is supposed to be nonempty.
To that end, we  study the asymptotic behaviour (as $t \to + \infty$) of the trajectories of the second-order differential equation 
\begin{equation}\label{edo01}
 \mbox{(AVD)}_{\alpha, \epsilon} \quad \quad \ddot{x}(t) + \frac{\alpha}{t} \dot{x}(t) + \nabla \Phi (x(t)) + \epsilon (t) x(t) =0,
\end{equation}
where $\alpha$  is a positive parameter,  and $\epsilon (t) x(t)$ is a
Tikhonov regularization term. Troughout the paper we assume that

\smallskip

\begin{center}
\textit{ $\epsilon : [t_0 , +\infty [ \to \mathbb R^+ $ is   a nonincreasing function, of class $\mathcal C^1$, and \  $\lim_{t \to \infty}\epsilon (t) =0$.}
\end{center}
  The system 
\begin{equation}\label{edo001}
 \mbox{(AVD)}_{\alpha} \quad \quad \ddot{x}(t) + \frac{\alpha}{t} \dot{x}(t) + \nabla \Phi (x(t)) =0,
\end{equation}
which corresponds to  the case $\epsilon (\cdot)\equiv 0$,  has been introduced by  Su, Boyd and 
Cand\`es in \cite{SBC}, 
as a continuous version of the  Nesterov accelerated method, see  \cite{Nest1}-\cite{Nest2}-\cite{Nest3}-\cite{Nest4}, and of the FISTA algorithm, see \cite{BT}.
For $\alpha \geq 3$, its trajectories satisfy the fast minimization property $ \Phi(x(t))-  \min_{\mathcal H}\Phi = \mathcal O (t^{-2})$, which is known to be the best 
possible estimate (in the worst case).
When $\alpha >3$, the weak convergence of the trajectories  was recently obtained by Attouch, Chbani, Peypouquet, and Redont in \cite{ACPR}, making the connection to the algorithmic results of Chambolle and Dossal \cite{CD}. We use the terminology introduced in \cite{ACPR}, where $\mbox{(AVD)}_{\alpha}$ stands for Asymptotic Vanishing Damping with parameter $\alpha$. Linking convergence of continuous dissipative systems and algorithms is an ongoing research topic, the  reader may consult \cite{HJ}, \cite{PB}, \cite{Pey}, \cite{Pey-Sor}.
Through the study of  $\mbox{(AVD)}_{\alpha, \epsilon} $, we seek to combine  the rapid optimization property of the system $\mbox{(AVD)}_{\alpha}$, with the property of strong convergence of trajectories to the solution of minimum norm. This latter  property   is typically attached to the Tikhonov approximation.

An abundant litterature has been devoted to the asymptotic hierarchical minimization property which results from the introduction of a vanishing viscosity term (in our context the Tikhonov approximation) in the dynamic. For first-order gradient systems and subdifferential inclusions, see \cite{AlvCab}, \cite{Att2},   \cite{AttCom}, \cite{AttCza2}, \cite{BaiCom} , \cite{CPS}, \cite{Hirstoaga}. 
Discrete time versions of these results provide algorithms combining  proximal based methods (for example forward-backward algorithms), with viscosity of penalization methods, see
\cite{AttCzaPey1}, \cite{AttCzaPey2}, \cite{BotCse}, \cite{Cab}, \cite{Hirstoaga}.

\medskip

A closely related dynamic to  $\mbox{(AVD)}_{\alpha, \epsilon} $  is the heavy ball with friction method  with a Tikhonov regularization term
\begin{equation}\label{HBF-Tikh}
  \mbox{(HBF)}_{\epsilon} \quad \quad \ddot{x}(t) + \gamma \dot{x}(t) + \nabla \Phi (x(t)) + \epsilon (t) x(t) =0.
\end{equation}
By contrast with $\mbox{(AVD)}_{\alpha, \epsilon} $, in $\mbox{(HBF)}_{\epsilon} $ the damping coefficient $\gamma$ 
is a fixed positive real number. The heavy ball with friction system, which corresponds to $\epsilon=0$ in $\mbox{(HBF)}_{\epsilon} $  is a dissipative dynamical system whose  optimization properties have been studied in detail in several articles, see \cite{AAC}, \cite{Al},     \cite{AA1}, \cite{aabr}, \cite{ACR}, \cite{AGR}, \cite{HJ}.
In \cite{AttCza1},  in the slow parametrization case $\int_0^{+\infty} \epsilon (t) dt = + \infty$, it is proved that  any solution $x(\cdot)$ of $\mbox{(HBF)}_{\epsilon} $  converges strongly to the minimum norm element of $\argmin \Phi$. But, without additional assumption on $\Phi$, no fast convergence result has been obtained for $\mbox{(HBF)}_{\epsilon}$.
 A parallel study has been developed  for PDE's, see \cite{AA} for damped hyperbolic equations with non-isolated equilibria, and 
\cite{AlvCab} for semilinear PDE's.  

As an original aspect of our approach (and a source of difficulties), through the study of $\mbox{(AVD)}_{\alpha, \epsilon}$ system,    we wish to  simultaneously handle the two vanishing parameters, the damping parameter and the Tikhonov parameter. Ideally, we want to achieve both rapid convergence and convergence towards the  minimum norm solution. As we shall see, this is a difficult task, since the two requirements are some way antagonistic.

Let us fix some $t_0 >0$, as a starting time. Taking $t_0 >0$ comes from the singularity of the damping coefficient $a(t) = \frac{\alpha}{t}$ at zero. Indeed, since we are only concerned about the asymptotic behaviour of the trajectories, we do not really care about the origin of time that is taken. If one insists starting from $t_0 =0$, then all the results remain valid  taking $a(t) = \frac{\alpha}{t+1}$.

Depending on the speed of convergence of 
$\epsilon (t)$ to zero, and the value of the positive parameter $\alpha$,   we analyze the convergence properties of the trajectories of $\mbox{(AVD)}_{\alpha, \epsilon}$. We obtain results ranging from the rapid convergence of the values when $\epsilon (t)$ decreases rapidly to zero, up to the convergence to the element of minimum norm of of $\argmin \Phi$, when $\epsilon (t)$ tends slowly to zero. Precisely,

\begin{description}
\item [A]
 In the "\textit{fast vanishing case}" $\int_{t_0}^{+\infty} \frac{\epsilon (t)}{t} dt <+ \infty$, just assuming that $\alpha >1$,
 we  show in Theorem \ref{Thm-weak-conv2} that, for any
 global solution trajectory of (\ref{edo01}), 
the following  minimizing property holds 
\begin{equation}\label{min01}
\lim_{t \rightarrow + \infty} \Phi (x(t)) = \inf_{\mathcal H} \Phi . 
\end{equation}
Under the stronger condition $\int_{t_0}^{+\infty} t \epsilon (t) dt < + \infty$,  assuming that $\alpha \geq 3$,  we  show  in Theorem \ref{tikh-fast} that any  global solution trajectory of (\ref{edo01}) satisfies the fast minimization property  
\begin{align}\label{basic-fast}
  \Phi(x(t))-  \min_{\mathcal H}\Phi \leq \frac{C}{t^2}.
\end{align}
When $\epsilon (\cdot) \equiv 0$, we recover  the fast convergence of the values  obtained by Su, Boyd and 
Cand\`es in \cite{SBC}. \\
When   $\alpha > 3$, in accordance with the convergence result of Attouch, Chbani, Peypouquet and Redont  \cite{ACPR}, we show that any global solution trajectory of (\ref{edo01})
converges weakly to a minimizer of $\Phi$. 
  
\medskip

\item [B] In the "\textit{slow vanishing case}" $\int_{t_0}^{+\infty} \frac{\epsilon (t)}{t} dt =+ \infty$, just assuming that $\alpha >1$, we show in 
Theorem \ref{Thm-strong-conv-Tikhonov} that for any global solution trajectory of (\ref{edo01}), the following ergodic convergence result holds
\begin{equation}\label{Tikh-002}
\lim_{t \rightarrow + \infty} \frac{1}{\int_{t_0}^t  \frac{\epsilon (\tau)}{\tau} } \int_{t_0}^t  \frac{\epsilon (\tau)}{\tau} \|x(\tau) - p \|  d \tau =0 , 
\end{equation}
where $p$ is the element of minimum norm of $\argmin \Phi$. Moreover
\begin{equation}\label{Tikh-02}
\liminf_{t \rightarrow + \infty} \|x(t) - p \| =0. 
\end{equation}
Convergence with a limit, instead of a lower limit, or a ergodic limit, is still an open puzzling question.

\end{description}

\section{Preliminary results and estimations.}

The existence of global solutions to \eqref{edo01} has been examined, for instance, in \cite[Proposition 2.2.]{CEG1} in the case of a general asymptotic vanishing damping coefficient, see also \cite{AttCza1} in the case of a fixed damping parameter. It is based on the formulation of \eqref{edo01} as a first-order system. Then apply Cauchy-Lipschitz theorem, and use energy estimates to pass from a local to a global solution. In our setting, for any $t_0 >0$, $\alpha >0$, and  $(x_0,v_0)\in \mathcal H\times\mathcal H$, there exists a unique global classical solution $x : [t_0, +\infty[ \rightarrow \mathcal H$ of \eqref{edo01}, satisfying the initial condition $x(t_0)= x_0$, $\dot{x}(t_0)= v_0$, under the sole assumption that $\inf\Phi > - \infty$. 

\bigskip

At different points, we shall use the {\it global energy} of the system, given by $W:[t_0,+\infty [\rightarrow\R$ 
\begin{equation}\label{E:W}
 W(t) :=   \frac{1}{2}\|\dot{x}(t)\|^2 + \Phi (x(t))
 + \frac{\epsilon (t)}{2} \|x((t)\|^2 .
\end{equation}

 After scalar multiplication of (\ref{edo01}) by $\dot{x}(t)$ we obtain
\begin{equation}\label{wconv-7}
 \frac{d}{dt} W(t) =   -\frac{\alpha}{t}\|\dot{x}(t)\|^2 
+ \frac{1}{2}\dot{\epsilon} (t)   \|x(t)\|^2. 
\end{equation}
By assumption $\dot{\epsilon} (t) \leq 0$, from which we deduce the following dissipative property.

\begin{lemma}\label{L:W}
Let $W$ be defined by \eqref{E:W}. For each $t>t_0$, we have
$$\frac{d}{dt} W(t) \leq -\frac{\alpha}{t}\|\dot{x}(t)\|^2.$$ 

i) Hence, $W$ is nonincreasing, and $W_\infty=\lim_{t\rightarrow+\infty}W(t)$ exists in $\R\cup\{-\infty\}$.  

\smallskip

ii) If $\Phi$ is bounded from below, then $W_\infty$ is finite, and 
\begin{equation}\label{wconv-7c}
 \int_{t_0}^{+\infty}\frac{1}{t}\|\dot{x}(t)\|^2 dt
\leq \frac{1}{\alpha} \left( W(t_0) - \inf_{\mathcal H} \Phi\right)  < + \infty.
\end{equation}
\end{lemma}

Now, given $z\in\mathcal H$, we define $h_z:[t_0,+\infty [ \to\R$ by
\begin{equation} \label{E:h_z}
h_z(t)=\frac{1}{2}\|x(t)-z\|^2.
\end{equation} 

By the Chain Rule, we have
$$\dot{h}_z(t) = \langle x(t) - z , \dot{x}(t)  \rangle\qquad\hbox{and}\qquad 
\ddot{h}_z(t) = \langle x(t) - z , \ddot{x}(t)  \rangle + \| \dot{x}(t) \|^2.$$
Using \eqref{edo01}, we obtain
\begin{equation} \label{wconv3}
\ddot{h}_z(t) + \frac{\alpha}{t} \dot{h}_z(t) =  \| \dot{x}(t) \|^2 + \langle x(t) - z , \ddot{x}(t) + \frac{\alpha}{t} \dot{x}(t)  \rangle =  \| \dot{x}(t) \|^2 -  \langle x(t) - z , \nabla \Phi (x(t)) + \epsilon (t) x(t)\rangle .
\end{equation} 
The convexity of $ \Phi$ implies
$$ \langle x(t) - z , \nabla \Phi (x(t))  \rangle \geq  \Phi (x(t)) - \Phi (z),$$
and we deduce that
\begin{equation} \label{E:ineq_h_z}
\ddot{h}_z(t) + \frac{\alpha}{t} \dot{h}_z(t) +\Phi (x(t)) - \Phi (z)  \leq \|\dot{x}(t)\|^2 - \epsilon (t)  \langle x(t) - z ,  x(t)\rangle.
\end{equation} 

Introducing $W$ in this expression we obtain the following differential inequality, that will play a central role in the convergence analysis of the trajectories of \eqref{edo01}.

\begin{lemma}\label{L:Wh_z}
Take $z\in \mathcal H$, and let $W$ and $h_z$ be defined by \eqref{E:W} and \eqref{E:h_z}, respectively. Then
\begin{equation} \label{E:ineq_h{_z}-ba}
\ddot{h}_z(t) + \frac{\alpha}{t} \dot{h}_z(t) + W(t) - \Phi (z)  \leq \frac{3}{2}\|\dot{x}(t)\|^2 + \frac{1}{2}\epsilon (t) \left(  2 \langle x(t),  z\rangle - \|x(t)\|^2\right)  .
\end{equation}
As a consequence, 
\begin{equation} \label{E:ineq_h{_z}-b}
\ddot{h}_z(t) + \frac{\alpha}{t} \dot{h}_z(t) + W(t) - \Phi (z)  \leq \frac{3}{2}\|\dot{x}(t)\|^2 + \frac{1}{2}\epsilon (t)  \|z\|^2  .
\end{equation}
\end{lemma}

\section{Fast vanishing case $\int_{t_0}^{+\infty} \frac{\epsilon (t)}{t} dt <+ \infty$} $\mbox{ }$
In the fast vanishing case $\int_{t_0}^{+\infty} \frac{\epsilon (t)}{t} dt <+ \infty$, we  obtain a minimization property, but  without any further information concerning the limiting behaviour of the trajectories.  Strengthening the hypothesis of rapid parametrization $\int_{t_{0}}^{+\infty} t\epsilon (t)dt < +\infty$, we  obtain a \textit{fast minimization} property, and  \textit{convergence of the  trajectories} to minimizers, but  without a precise identification of the limit.  Asymptoticaly, the regularizing term is not large enough  to induce a  viscosity selection.
Note that  despite  the fact that it is not active asymptotically, the Tikhonov term makes the dynamic governed at all times $t$ by a strongly monotone operator, which induces favorable numerical aspects.

\subsection{Minimizing property}

Let us examine the
minimizing property of the trajectories, $\lim_{t \rightarrow + \infty} \Phi (x(t)) = \inf_{\mathcal H} \Phi$.
\begin{Theorem} \label{Thm-weak-conv2}
Let $\Phi : \mathcal H \rightarrow \mathbb R$ be a convex continuously differentiable function 
such that   $\inf_{\mathcal H} \Phi > - \infty$ (the set $\argmin \Phi$ is possibly empty).
Suppose that $\alpha >1$. Let $\epsilon : [t_0 , +\infty [
\to \mathbb R^+ $ be a $\mathcal C^1$ decreasing function such that
\begin{equation}\label{basic-epsilon1}
\int_{t_0}^{+\infty} \frac{\epsilon (t)}{t} dt <+ \infty .
\end{equation}
Let $x(\cdot)$ be a  classical global solution  of $ \mbox{(AVD)}_{\alpha, \epsilon} $. 
Then, the following  minimizing property holds 
\begin{equation}\label{min1}
\lim_{t \rightarrow + \infty} \Phi (x(t)) = \inf_{\mathcal H} \Phi, 
\end{equation}
and
\begin{equation}\label{min10}
\lim_{t \to +\infty}  \|\dot{x}(t)\| = 0 .
\end{equation}
\end{Theorem}

\begin{proof}
Let $z\in \mathcal H$, taken arbitrary. By Lemma \ref{L:Wh_z}, 
the function  $h: [t_0, +\infty[ \rightarrow \mathbb R^+$ defined by
$ h(t) = \frac{1}{2}\| x(t) - z\|^2$
 satisfies the differential inequality
\begin{equation}\label{wconv-8b}
 \ddot{h}(t) + \frac{\alpha}{t} \dot{h}(t) + W(t) - \Phi (z)  \leq g(t)
\end{equation}
where
\begin{equation}\label{wconv-8c}
g(t):= \frac{3}{2}\|\dot{x}(t)\|^2 + \frac{\epsilon (t)}{2} \|z\|^2  .
\end{equation}
 Let us  proceed with the integration of the differential  inequality (\ref{wconv-8b}).
After multiplication of
 (\ref{wconv-8b}) by $t^{\alpha}$ we have
\begin{equation*}
 \frac{d}{dt}\left( t^{\alpha}  \dot{h}(t)\right)  + t^{\alpha} \left(  W(t) - \Phi (z) \right)    \leq t^{\alpha} g(t) .
\end{equation*}
Integrating from $t_0$ to $t$ we obtain
\begin{equation*}
t^{\alpha}  \dot{h}(t) - t_0^{\alpha}  \dot{h}(t_0) 
+ \int_{t_0}^t  s^{\alpha} \left(  W(s) - \Phi (z) \right) ds
\leq \int_{t_0}^t  s^{\alpha} g(s) ds.
\end{equation*}
Since the function $W(\cdot)$ is nonincreasing (see Lemma \ref{L:W}), we deduce that
\begin{equation*}
t^{\alpha}  \dot{h}(t) - t_0^{\alpha}  \dot{h}(t_0) 
+ \left(  W(t) - \Phi (z) \right)\int_{t_0}^t  s^{\alpha}  ds
\leq \int_{t_0}^t  s^{\alpha} g(s) ds,
\end{equation*}
which, after division by $t^{\alpha}$, gives
\begin{equation*}
  \dot{h}(t) 
+ t^{-\alpha}\left(  W(t) - \Phi (z) \right)\int_{t_0}^t  s^{\alpha}  ds
\leq \frac{ t_0^{\alpha}  \dot{h}(t_0)}{t^{\alpha}} +
t^{-\alpha}\int_{t_0}^t  s^{\alpha} g(s) ds.
\end{equation*}
Integrating once more from $t_0$ to $t$ we obtain
\begin{equation*}
  h(t) - h(t_0) 
+ \int_{t_0}^t  s^{-\alpha}\left(  W(s) - \Phi (z) \right)\left(\int_{t_0}^s  {\tau}^{\alpha}  d\tau \right) ds 
\leq  t_0^{\alpha}  \dot{h}(t_0)\int_{t_0}^t {s^{-\alpha}}ds +
\int_{t_0}^t  s^{-\alpha} \left( \int_{t_0}^s  {\tau}^{\alpha} g(\tau) d\tau \right)ds .
\end{equation*}
Using again that $W(\cdot)$ is nonincreasing, we deduce that
\begin{equation*}
  h(t) - h(t_0) 
+ \left(W(t) - \Phi (z) \right ) \int_{t_0}^t  s^{-\alpha}\left(\int_{t_0}^s  {\tau}^{\alpha}  d\tau \right) ds 
\leq  \frac{1}{\alpha -1} t_0  |\dot{h}(t_0)| +
\int_{t_0}^{t}  s^{-\alpha} \left( \int_{t_0}^s  {\tau}^{\alpha} g(\tau)  d\tau \right)ds .
\end{equation*}
Computing the  first integral, and since $h$ is nonnegative, we obtain
\begin{equation}\label{wconv-16}
\frac{1}{\alpha + 1}\left(W(t) - \Phi (z) \right ) \left(              \frac{t^2}{2} - \frac{{t_0}^2}{2} +   \frac{{t_0}^{\alpha +1}}{(\alpha - 1)t^{\alpha -1} }  - \frac{{t_0}^2}{\alpha -1}       \right) 
\leq  h(t_0) + \frac{1}{\alpha -1} t_0  |\dot{h}(t_0)| +
\int_{t_0}^{t}  s^{-\alpha} \left( \int_{t_0}^s  {\tau}^{\alpha} g(\tau) d\tau \right)ds .
\end{equation}
Let us compute this last integral by Fubini's theorem
\begin{align*}
\int_{t_0}^{t}  s^{-\alpha} \left( \int_{t_0}^s  {\tau}^{\alpha} g(\tau)  d\tau \right)ds &=
\int_{t_0}^{t}   \left( \int_{\tau}^t  s^{-\alpha}ds  \right){\tau}^{\alpha} g(\tau)  d\tau  \\
&   = \frac{1}{\alpha -1} \int_{t_0}^{t}  \left(       \frac{1}{{\tau}^{\alpha -1}} - \frac{1}{{t}^{\alpha -1}}\right)
{\tau}^{\alpha} g(\tau)  d\tau  \\
&   \leq \frac{1}{\alpha -1} \int_{t_0}^{t}  
\tau g(\tau)  d\tau .
\end{align*}
Returning to (\ref{wconv-16}), we deduce that
\begin{equation}\label{wconv-18}
\frac{1}{\alpha + 1}\left(W(t) - \Phi (z) \right ) \left(              \frac{t^2}{2} - \frac{{t_0}^2}{2} +   \frac{{t_0}^{\alpha +1}}{(\alpha - 1)t^{\alpha -1} }  - \frac{{t_0}^2}{\alpha -1}       \right) 
\leq h(t_0)+ \frac{1}{\alpha -1} t_0  |\dot{h}(t_0)| +
\frac{1}{\alpha -1} \int_{t_0}^{t}  
\tau g(\tau)  d\tau .
\end{equation}
By (\ref{wconv-7c}) we have 
$\int_{t_0}^{\infty} \frac{1}{t}\|\dot{x}(t)\|^2 dt <+\infty$. 
By assumption
$\int_{t_0}^{+\infty} \frac{1}{t} \epsilon (t)dt <+\infty $.
Hence, by definition  (\ref{wconv-8c}) of $g$, we have
\begin{equation}\label{wconv-71}
\int_{t_0}^{\infty} \frac{1}{t}g(t) dt <+\infty. 
\end{equation}
Let us rewrite (\ref{wconv-18}) as 
\begin{equation*}
\frac{1}{\alpha + 1}\left(W(t) - \Phi (z) \right ) \left(              \frac{t^2}{2} - \frac{{t_0}^2}{2} +   \frac{{t_0}^{\alpha +1}}{(\alpha - 1)t^{\alpha -1} }  - \frac{{t_0}^2}{\alpha -1}       \right) 
\leq  h(t_0)+\frac{1}{\alpha -1} t_0  |\dot{h}(t_0)| +
\frac{1}{\alpha -1} \int_{t_0}^{t}  
{\tau}^2  \frac{1}{\tau}g(\tau)  d\tau .
\end{equation*}
Dividing by $t^2$, and letting $t \to + \infty$, we obtain thanks to Lemma \ref{basic-int}
\begin{equation}\label{wconv-20}
\limsup_{t \to +\infty} W(t) \leq \Phi (z).
\end{equation}
Since $W(t) \geq \Phi (x(t))$, we deduce that
$$
\limsup_{t \to +\infty} \Phi (x(t)) \leq \Phi (z).
$$
This being true for any $z\in \mathcal H$, we obtain
$$
\limsup_{t \to +\infty}  \Phi (x(t)) \leq \inf \Phi .
$$
The other inequality $\liminf \Phi (x(t)) \geq \inf \Phi $ being trivially satisfied, we finally obtain
$$
\lim_{t \to +\infty}  \Phi (x(t)) = \inf_{\mathcal H} \Phi .
$$
Returning to (\ref{wconv-20}), we also obtain
$$
\lim_{t \to +\infty}  \|\dot{x}(t)\| = 0 .
$$
\end{proof}

\begin{remark}
With $t_0 >0$, the condition
$\int_{t_0}^{+\infty} \frac{\epsilon (t)}{t} dt <+ \infty $
is satisfied by
$
\epsilon (t) = \frac{1} {t^{\gamma}}
$
for any $\gamma >0$, and by $
\epsilon (t) = \frac{1} {(\ln t)^{\gamma}}
$
for any $\gamma >1$.
The property
$
\int_{t_0}^{+\infty} \frac{\epsilon (t)}{t} dt = + \infty 
$
is satisfied  by
$
\epsilon (t) = \frac{1}{({\ln t})^{\gamma}},
$
 for $0 < \gamma \leq 1$.
\end{remark}

\subsection{Case $\int_{t_{0}}^{+\infty} t\epsilon (t)dt < +\infty$: fast minimization}

The  following fast minimization property, and the convergence of trajectories, is consistent with the results obtained in \cite{ACPR} for the perturbed version of $ \mbox{(AVD)}_{\alpha}$. The Tikhonov regularization term acts as a small perturbation which does not affect the convergence properties of $ \mbox{(AVD)}_{\alpha}$.
\begin{Theorem}\label{tikh-fast} 
Let $\Phi : \mathcal H \rightarrow \mathbb R$ be a convex continuously differentiable function 
such that  $\argmin \Phi$ is nonempty.
 Let $\epsilon : [t_0 , +\infty [
\to \mathbb R^+ $ be a $\mathcal C^1$ nonincreasing  function such that
$\int_{t_{0}}^{+\infty} t\epsilon (t)dt < +\infty$. 
Let $x(\cdot)$ be a  classical  global solution  of $ \mbox{(AVD)}_{\alpha, \epsilon} $.

\smallskip

i) Suppose $\alpha \geq 3$. Then the following fast convergence of the values holds true
$$
 \Phi (x(t))-\inf \Phi\leq\dfrac{C}{t^{2}}  .
$$
Moreover
$$
\int_{t_{0}}^{+\infty}t\epsilon(t)\Vert x(t)\Vert^{2}<+\infty .
$$
\indent ii) Suppose $\alpha > 3$. Then $x(t)$   converges weakly to an element of \  $\argmin \Phi$, as $t \to + \infty$. Moreover,
\begin{align}
& \int_{t_{0}}^{+\infty}t(\Phi (x(t))-\inf \Phi)dt<+\infty ,\label{energ-est-1}\\
& \int_{t_{0}}^{+\infty} t \|\dot{x}(t)\|^2dt<+\infty ,\label{energ-est-2}\\
 & \sup_{t \geq t_0} t  \|\dot{x}(t)\| < + \infty \label{energ-est-3}. 
\end{align}
\end{Theorem}
\begin{proof} \textit{i)} The proof is parallel to that of \cite{ACPR}.
Fix  $z\in\argmin \Phi$, and consider the energy function 
$$\mathcal{E}(t)=\dfrac{2}{\alpha -1}t^{2}\left[ f_{t}(x(t))-\inf \Phi\right] +(\alpha -1)
\Vert x(t)-z+\frac{t}{\alpha -1}\dot{x}(t)\Vert^{2},$$
where $f_t : \mathcal H \to \mathbb R$ is defined for any $x\in \mathcal H$ by 
\begin{equation}\label{Tikh-4b0}
 f_t (x) := \Phi(x) + \frac{\epsilon (t)}{2}\| x \|^2 .
\end{equation}
Let us observe  that
\begin{align*}
\nabla f_t (x(t)) &= \nabla\Phi(x(t)) + \epsilon (t) x(t)\\
&= -\ddot{x}(t) - \frac{\alpha}{t} \dot{x}(t) .
\end{align*}
By derivating $\mathcal{E}(\cdot)$, and using the above relation we obtain
\begin{align*}
\dot{\mathcal{E}}(t)&=\dfrac{4t}{\alpha -1}\left[ f_{t}(x(t))-\inf \Phi\right]+ \dfrac{2}{\alpha -1}t^{2}\left[ \langle \nabla f_{t}(x(t)),\dot{x}(t)\rangle +\dot{\epsilon}(t)\frac{\Vert x(t)\Vert^{2}}{2}\right]\vspace{2mm} \\
& \qquad \qquad \qquad \qquad \qquad \qquad +2t\langle x(t)-z +\frac{t}{\alpha -1}\dot{x}(t),\ddot{x}(t) + \frac{\alpha}{t} \dot{x}(t)\rangle\\
& = \dfrac{4t}{\alpha -1}\left[ f_{t}(x(t))-\inf \Phi\right]+ \dfrac{2}{\alpha -1}t^{2}\left[ \langle \nabla f_{t}(x(t)),\dot{x}(t)\rangle +\dot{\epsilon}(t)\frac{\Vert x(t)\Vert^{2}}{2}\right]\vspace{2mm} \\
&\qquad \qquad \qquad \qquad \qquad \qquad -2t\langle x(t)-z +\frac{t}{\alpha -1}\dot{x}(t),\nabla f_{t}(x(t))\rangle .
\end{align*}
After simplification, we obtain
\begin{equation}\label{Lip-E1}
\dot{\mathcal{E}}(t)=\dfrac{4t}{\alpha -1}\left[ f_{t}(x(t))-\inf \Phi\right]+\dfrac{t^{2}}{\alpha -1}\dot{\epsilon}(t)\Vert x(t)\Vert^{2}-2t\langle x(t) -z,\nabla f_{t}(x(t))\rangle .
\end{equation}
By the strong convexity of $f_t$
$$f_{t}(z)-f_{t}(x(t))\geq \langle\nabla f_{t}(x(t)),z-x(t)\rangle+\frac{\epsilon(t)}{2}\Vert x(t)-z\Vert^{2}.$$
Equivalently
\begin{equation}\label{Lip-E2}
 \langle\nabla f_{t}(x(t)), x(t)-z \rangle \geq f_{t}(x(t)) - \inf \Phi - \frac{\epsilon(t)}{2} \|z\|^2     +\frac{\epsilon(t)}{2}\Vert x(t)-z\Vert^{2}.
 \end{equation}
Combining (\ref{Lip-E1}) with (\ref{Lip-E2}) we obtain
$$\dot{\mathcal{E}}(t)+(2-\dfrac{4}{\alpha -1})t\left[ f_{t}(x(t))-\inf \Phi\right]-\dfrac{t^{2}}{\alpha -1}\dot{\epsilon}(t)\Vert x(t)\Vert^{2}+t\epsilon(t)\Vert x(t)-z\Vert^{2} \leq t\epsilon(t) \|z\|^2.$$
We exploit the expression of $f_{t}$ to write
\begin{equation}\label{Lip-E3}
\dot{\mathcal{E}}(t)+2(\dfrac{\alpha -3}{\alpha -1})t\left[ \Phi (x(t))-\inf \Phi\right]+\left[(\alpha-3)\epsilon(t)-t\dot{\epsilon}(t)\right]\dfrac{t\Vert x(t)\Vert^{2}}{\alpha-1}+t\epsilon(t)\Vert x(t)-z\Vert^{2}\leq t\epsilon(t) \|z\|^2 .
\end{equation}
Since  $\epsilon (\cdot)$ is a  nonincreasing, nonnegative  function, and $\alpha \geq 3$, we infer
$$\dot{\mathcal{E}}(t)\leq t\epsilon(t) \|z\|^2 .$$
From $\int_{t_{0}}^{+\infty} t\epsilon (t)dt < +\infty$
we deduce that
the positive part $[\dot{\mathcal{E}}]_+$ of $\dot{\mathcal{E}}$ belongs to $L^1(t_0,+\infty)$. Since $\mathcal{E}$ is bounded from below, 
it follows that $\mathcal{E}(t)$ has a limit as $t \to +\infty$, and hence is bounded, which gives the claim
$ \Phi (x(t))-\inf \Phi\leq\dfrac{C}{t^{2}}$.
Now integrating (\ref{Lip-E3}), we obtain 
\begin{equation}\label{Lip-E31}
\int_{t_{0}}^{+\infty}t\epsilon(t)\Vert x(t)-z\Vert^{2}  dt <+\infty .
\end{equation}
Combining the inequality
$$t\epsilon(t)\Vert x(t)\Vert^{2} \leq 2t\epsilon(t)\Vert x(t)-z\Vert^{2} + 2t\epsilon(t)\Vert z\Vert^{2}$$ 
with (\ref{Lip-E31}) and $\int_{t_{0}}^{+\infty} t\epsilon (t)dt < +\infty$, we  conclude that
\begin{equation}\label{Lip-E32}
\int_{t_{0}}^{+\infty}t\epsilon(t)\Vert x(t)\Vert^{2} dt <+\infty .
\end{equation}

\medskip

\textit{ii)} Let us now suppose $\alpha >3$.
By integration of (\ref{Lip-E3}) we first obtain
\begin{equation}\label{Lip-E4}
\int_{t_{0}}^{+\infty}t(\Phi (x(t))-\inf \Phi)dt<+\infty.
\end{equation}
In order to obtain further energy estimates, let us take the scalar product of  (\ref{edo01}) with 
$t^2 \dot{x}(t)$. We obtain
$$
\frac{t^2}{2} \frac{d}{dt}\|\dot{x}(t)\|^2
+ \alpha t \|\dot{x}(t)\|^2 + t^2 \frac{d}{dt} \Phi (x(t)) +  \frac{t^2}{2}\epsilon (t) \frac{d}{dt}\|x(t)\|^2  =0 .
$$
After integration on $[t_0, t]$, we obtain
\begin{align*}
\frac{t^2}{2} \|\dot{x}(t)\|^2 + (\alpha -1) \int_{t_0}^t s \|\dot{x}(s)\|^2 ds \ + & \  t^2 \left(  \Phi (x(t)) -\inf \Phi \right) -2 \int_{t_0}^t s \left(  \Phi (x(s)) -\inf \Phi \right) ds \\
&+ \frac{t^2}{2}\epsilon (t) \|x(t)\|^2 - \int_{t_0}^t  (s \epsilon (s) + \frac{s^2}{2}  \dot{\epsilon}(s)  ) \|x(s)\|^2 ds \leq C
\end{align*}
where $C$ is independent of $t$, and just depends on the initial data. By using  $\dot{\epsilon}(\cdot) \leq 0$, we deduce that 
$$
\frac{t^2}{2} \|\dot{x}(t)\|^2 + (\alpha -1) \int_{t_0}^t s \|\dot{x}(s)\|^2 ds  \\
  \leq C + 2 \int_{t_0}^{\infty} s\left(  \Phi (x(s)) -\inf \Phi \right) ds + \int_{t_0}^{\infty}  s \epsilon (s)  \|x(s)\|^2 ds.
$$
Using the previous estimates (\ref{Lip-E32}) and (\ref{Lip-E4}), and since $\alpha >1$, we deduce that
$$
\int_{t_{0}}^{+\infty}t  \|\dot{x}(t)\|^2 dt <  +\infty
$$
and
$$
\sup_{t \geq t_0} t  \|\dot{x}(t)\| < + \infty .
$$
We now have all the ingredients to prove the weak convergence of trajectories. From Lemma
\ref{L:Wh_z}, (\ref{E:ineq_h{_z}-b}), we have
\begin{equation}
t\ddot{h}_z(t) + \alpha \dot{h}_z(t)   \leq  g(t) ,
\end{equation}
with 
$$g(t)= \frac{3}{2}t\|\dot{x}(t)\|^2 + \frac{1}{2}t\epsilon (t)  \|z\|^2 .$$
 Using again the assumption  
$\int_{t_{0}}^{+\infty} t\epsilon (t)dt < +\infty$, and 
the energy estimate (\ref{energ-est-2})
$\int_{t_{0}}^{+\infty} t \|\dot{x}(t)\|^2dt<+\infty$ ,  it follows that $g\in L^1 (t_0, +\infty)$.
 Lemma \ref{basic-edo} now shows that 
 the positive part $[\dot h]_+$ of $\dot h$ belongs to $L^1(t_0,+\infty)$, and $\lim_{t\to+\infty}h(t)$ exists.
This is one of the two conditions of Opial's lemma \ref{Opial}.
The other condition is clearly satisfied: we know that $\Phi (x(t))$ tends to the infimal value of $\Phi$. By the lower semicontinuity of $\Phi$ for the weak topology,  every weak cluster point of of $x(\cdot)$ is a minimizer of $\Phi$.
\end{proof}

\begin{remark} The above convergence result is not  a consequence of the general perturbation Theorem of \cite{ACPR}, because we don't know a priori if the trajectory remains bounded. This is only obtained ultimately as a consequence of the convergence of the trajectory for the weak topology.
\end{remark}


\section{Slow vanishing case: $\int_{t_0}^{+\infty} \frac{\epsilon (t)}{t} dt = + \infty$} 

\smallskip

In this section,   $S=\argmin \Phi$ is supposed to be nonempty.
When $\epsilon (t)$ does not tend to zero too rapidly, a natural conjecture is that any orbit  of $ \mbox{(AVD)}_{\alpha, \epsilon} $ converges strongly to the element of $\argmin \Phi$ which has minimal norm.
 To analyze this delicate question, we first recall some classical facts concerning the Tikhonov approximation.
 For each $\epsilon >0$, we denote by $x_{\epsilon}$  the unique solution  of  the strongly convex minimization problem
$$
x_{\epsilon} = \argmin_{x \in \mathcal H} \left\lbrace  \Phi (x) + \frac{\epsilon}{2}\| x \|^2 \right\rbrace .
$$
Equivalently,
$$
\nabla  \Phi (x_{\epsilon}) + \epsilon x_{\epsilon} =0.
$$
Let us recall that  the Tikhonov approximation curve, $\epsilon \mapsto x_{\epsilon}$, satisfies the well-known strong convergence property:
$$
\lim _{\epsilon \to 0}x_{\epsilon} =   p , 
$$
where $p$ is the element of minimal norm of the closed convex nonempty set $\argmin \Phi$. This result was first  obtained by Tikhonov and Arsenin \cite{TA} in the case of ill-posed least square problems, and Browder \cite{Browder} for monotone variational inequalities, then extended and revisited by many authors, see for example \cite[Corollary 5.2]{Att2},  \cite[Theorem 23.44]{BC}, \cite{Reich}. Moreover, by the monotonicity property of $\nabla  \Phi $, and $\nabla  \Phi (p) =0$, $\nabla  \Phi (x_{\epsilon}) = - \epsilon x_{\epsilon}$, we have
$$
\langle x_{\epsilon} -   p ,  -\epsilon x_{\epsilon} \rangle  \geq 0
$$ 
which, after dividing by $\epsilon >0$, and by Cauchy-Schwarz inequality gives 
\begin{equation}\label{Tikh-0}
 \|  x_{\epsilon}  \| \leq \|p \| \quad \mbox{for all} \ \epsilon >0.
\end{equation}   
The following result gives a partial answer to the strong convergence property of trajectories of $\mbox{(AVD)}_{\alpha, \epsilon} $ to the solution with minimal norm.

\begin{Theorem} \label{Thm-strong-conv-Tikhonov}
Let $\Phi : \mathcal H \rightarrow \mathbb R$ be a convex continuously differentiable function 
such that  $\argmin \Phi$ is nonempty.
Suppose that $\alpha >1$. Let $\epsilon : [t_0 , +\infty [
\to \mathbb R^+ $ be a $\mathcal C^1$ decreasing function such that $\lim_{t \to +\infty} \epsilon (t) =0$, and  such that
\begin{equation}\label{Tikh-1}
\int_{t_0}^{+\infty} \frac{\epsilon (t)}{t} dt =+ \infty .
\end{equation}
Let $x(\cdot)$ be a  classical global solution  of {\rm(\ref{edo001})}. 
Then, the following ergodic convergence property is satisfied
\begin{equation}\label{Tikh-22}
\lim_{t \rightarrow + \infty} \frac{1}{\int_{t_0}^t  \frac{\epsilon (\tau)}{\tau} d\tau} \int_{t_0}^t  \frac{\epsilon (\tau)}{\tau} \|x(\tau) - p \|  d \tau =0 , 
\end{equation}
where $p$ is the element of minimal norm of $\argmin \Phi$.
Moreover, 
\begin{equation}\label{Tikh-2}
\liminf_{t \rightarrow + \infty} \|x(t) - p \| =0 .
\end{equation}
\end{Theorem}

\begin{proof}
Our proof is an adaptation to our situation of the argument developed by Cominetti-Peypouquet-Sorin in \cite{CPS}. Let $p := \mbox{proj}_{\argmin \Phi} 0$ \ be the unique element of minimal norm of the closed convex nonempty set $\argmin \Phi$, and set 
\begin{equation}\label{Tikh-3}
 h(t) = \frac{1}{2}\| x(t) - p\|^2 .
\end{equation}
By a similar computation as in Theorem \ref{Thm-weak-conv2}
\begin{equation}\label{Tikh-4}
 \ddot{h}(t) + \frac{\alpha}{t} \dot{h}(t) =  \| \dot{x}(t) \|^2 + \langle x(t) - p , \ddot{x}(t) + \frac{\alpha}{t} \dot{x}(t)  \rangle .
\end{equation}
We use the function $f_t$ introduced in (\ref{Tikh-4b0})
\begin{equation*}
x \mapsto f_t (x) := \Phi(x) + \frac{\epsilon (t)}{2}\| x \|^2,
\end{equation*}
 and observe  that
\begin{align*}
\nabla f_t (x(t)) &= \nabla\Phi(x(t)) + \epsilon (t) x(t)\\
&= -\ddot{x}(t) - \frac{\alpha}{t} \dot{x}(t) .
\end{align*}
By the strong convexity property of $f_t$, and the above relation, we infer
\begin{align*}
 f_t (p) &\geq   f_t (x(t)) + \langle \nabla f_t (x(t)) , p- x(t)  \rangle  + \frac{\epsilon (t)}{2}\| p- x(t) \|^2 \\
 & \geq f_t (x(t)) + \langle -\ddot{x}(t) - \frac{\alpha}{t} \dot{x}(t) , p- x(t)  \rangle  + \frac{\epsilon (t)}{2}\| p- x(t) \|^2 .
\end{align*}
Equivalently
\begin{equation}\label{Tikh-7}
\langle x(t) - p , \ddot{x}(t) + \frac{\alpha}{t} \dot{x}(t)  \rangle + \epsilon (t) h(t) \leq  f_t (p) -   f_t (x(t)) .
\end{equation}
By definition of $ x_{\epsilon}$, we have 
\begin{align}\label{Tikh-8}
f_t (x_{\epsilon(t)}) & = \Phi(x_{\epsilon(t)}) + \frac{\epsilon(t)}{2}\| x_{\epsilon(t)} \|^2 \\
& \leq \Phi(x(t)) + \frac{\epsilon(t)}{2} \| x(t)\|^2 = f_t (x(t)) .
\end{align}
Combining (\ref{Tikh-7}) and (\ref{Tikh-8}) we obtain 
\begin{equation}\label{Tikh-9}
\langle x(t) - p , \ddot{x}(t) + \frac{\alpha}{t} \dot{x}(t)  \rangle + \epsilon (t) h(t) \leq  f_t (p) -   f_t (x_{\epsilon(t)}).
\end{equation}
Since $\Phi (p) \leq \Phi (x_{\epsilon(t)})$, we have
\begin{align}\label{Tikh-10}
f_t (p) -   f_t (x_{\epsilon(t)})&= \Phi(p) + \frac{\epsilon (t)}{2}\| p \|^2 - \Phi(x_{\epsilon(t)}) - \frac{\epsilon (t)}{2}\| x_{\epsilon(t)} \|^2 \\
&\leq \frac{\epsilon (t)}{2} ( \| p \|^2 - \| x_{\epsilon(t)} \|^2   ).
\end{align}
Combining (\ref{Tikh-9}) and (\ref{Tikh-10}), we get
\begin{equation}\label{Tikh-11}
\langle x(t) - p , \ddot{x}(t) + \frac{\alpha}{t} \dot{x}(t)  \rangle + \epsilon (t) h(t) \leq  \frac{\epsilon (t)}{2} ( \| p \|^2 - \| x_{\epsilon(t)} \|^2   ).
\end{equation}
Returning to (\ref{Tikh-4}) we obtain
 \begin{equation}\label{Tikh-12}
 \ddot{h}(t) + \frac{\alpha}{t} \dot{h}(t) + \epsilon (t) h(t) \leq  \| \dot{x}(t) \|^2  + \frac{\epsilon (t)}{2} ( \| p \|^2 - \| x_{\epsilon(t)} \|^2   ).
\end{equation}
Equivalently
 \begin{equation}\label{Tikh-13}
  \epsilon (t) h(t) \leq  \| \dot{x}(t) \|^2  + \frac{\epsilon (t)}{2} ( \| p \|^2 - \| x_{\epsilon(t)} \|^2   ) - \frac{1}{t^{\alpha}} \frac{d}{dt}(t^{\alpha}\dot{h} (t)  ).
\end{equation}
After dividing by $t$
 \begin{equation}\label{Tikh-14}
 \frac{ \epsilon (t)}{t} h(t) \leq   \frac{1}{t} \| \dot{x}(t) \|^2  + \frac{\epsilon (t)}{2t} ( \| p \|^2 - \| x_{\epsilon(t)} \|^2   ) - \frac{1}{t^{\alpha +1}} \frac{d}{dt}(t^{\alpha}\dot{h} (t)  ).
\end{equation}
Set
$$
\delta (t) :=  \frac{1}{2}(\| p \|^2 - \| x_{\epsilon(t)} \|^2 ) ,
$$
which by (\ref{Tikh-0}) is nonnegative,  and by the strong convergence property of the Tikhonov approximation satisfies
$$
\lim_{t\to +\infty} \delta (t) =0.
$$
Let us integrate (\ref{Tikh-14}) on $[t_0, t]$. We obtain the existence of some positive constant $C$ such that for all $t\geq t_0$
\begin{equation}\label{Tikh-15}
\int_{t_0}^t  \frac{\epsilon (\tau)}{\tau} ( h(\tau) -    \delta (\tau) ) d \tau \leq C .  
\end{equation}
From $\int_{t_0}^{+\infty} \frac{\epsilon (t)}{t} dt = + \infty$, we deduce that 
$$
\liminf_{t \rightarrow + \infty} \ ( h(t) -    \delta (t) ) \leq 0
$$
and since $\lim_{t\to +\infty} \delta (t) =0$, we  obtain
$$
\liminf_{t \rightarrow + \infty} \  h(t) =0 .
$$
Let us now prove  strong ergodic convergence to the solution with minimal norm. We start from
\begin{align}\label{Tikh-16}
\int_{t_0}^t  \frac{\epsilon (\tau)}{\tau} h(\tau)  d \tau \ &= \int_{t_0}^t  \frac{\epsilon (\tau)}{\tau} ( h(\tau) -    \delta (\tau) ) d \tau + \int_{t_0}^t  \frac{\epsilon (\tau)}{\tau} \delta(\tau)  d \tau \\
&\leq C + \int_{t_0}^t  \frac{\epsilon (\tau)}{\tau} \delta(\tau)  d \tau,  \nonumber 
\end{align}
where the second inequality comes from (\ref{Tikh-15}).
After dividing (\ref{Tikh-16}) by $\int_{t_0}^t  \frac{\epsilon (\tau)}{\tau} d\tau$,  using that 
$\lim_{t\to +\infty} \int_{t_0}^t  \frac{\epsilon (\tau)}{\tau}  =+\infty$, and
$\lim_{t\to +\infty} \delta (t) =0$,  we obtain
\begin{align}\label{Tikh-17}
\limsup_{t \rightarrow + \infty} \frac{1}{\int_{t_0}^t  \frac{\epsilon (\tau)}{\tau} d\tau} \int_{t_0}^t  \frac{\epsilon (\tau)}{\tau} \|x(\tau) - p \|  d \tau \ \leq 0.  
\end{align}
Since $h$ is nonnegative, this gives  the ergodic convergence result 
\begin{equation*}
\lim_{t \rightarrow + \infty} \frac{1}{\int_{t_0}^t  \frac{\epsilon (\tau)}{\tau} d\tau} \int_{t_0}^t  \frac{\epsilon (\tau)}{\tau} \|x(\tau) - p \|  d \tau =0 . 
\end{equation*}
By using Jensen's inequality, we deduce that
\begin{equation*}
\lim_{t \rightarrow + \infty} \frac{1}{\int_{t_0}^t  \frac{\epsilon (\tau)}{\tau} d\tau} \int_{t_0}^t  \frac{\epsilon (\tau)}{\tau} x(\tau)  d \tau =p . 
\end{equation*}

\end{proof}
\section{An illustrative example}

 Let us examine a simple situation where we are able to compute explicitely  the  trajectories of  $\mbox{(AVD)}_{\alpha, \epsilon}$, and hence analyze their  asymptotic behavior.
We  use a symbolic differential computation software to determine explicit solutions for $\mbox{(AVD)}_{\alpha, \epsilon}$ in terms of classical functions, and Bessel functions of first and second type. We used WolframAlpha  Computational Knowledge Engine, available at {\tt http://www.wolframalpha.com.}
 
\medskip
 
 Let $\Phi:\R\to\R$ be the function which is identically zero, i.e., $\Phi(x)=0$ for all $x\in \mathbb R$. 
The $\mbox{(AVD)}_{\alpha} $ system (without Tikhonov regularization term) writes
\begin{equation}\label{edo001-example}
 \mbox{(AVD)}_{\alpha} \quad \quad \ddot{x}(t) + \frac{\alpha}{t} \dot{x}(t)  =0.
\end{equation} 
An elementary computation shows that, for any $\alpha >1$, each trajectory of (\ref{edo001-example}) converges. Its limit is equal to $x(t_0) + \frac{t_0}{\alpha -1}\dot{x}(t_0)$, which depends on the the initial data.
 
Let us now examine the convergence properties of the trajectories of the corresponding $\mbox{(AVD)}_{\alpha, \epsilon}$ system,
\begin{equation}\label{edo01-example}
 \mbox{(AVD)}_{\alpha, \epsilon} \quad \quad \ddot{x}(t) + \frac{\alpha}{t} \dot{x}(t) + \epsilon(t) x(t) =0,
\end{equation}
 which includes a Tikhonov regularization term.
Note that the set of minimizers of $\Phi$ is the whole real line, whose  minimum norm element is precisely zero. Since the convergence of  values is trivially satisfied in this case, the only relevant question is to examine  the convergence of  trajectories toward zero. In all the following examples we take as Cauchy data $x(1)= 1$ and $\dot{x}(1)=0$.

\medskip

1. \textit{Case $\epsilon (t) = \frac{1}{1+ \ln t}$.}
 The system writes
\begin{equation}\label{algo7bba}
\left\{
\begin{array}{rcl}
 & \ \ddot{x}(t) + \displaystyle{\frac{\alpha}{t}} \dot{x}(t) + \displaystyle{\frac{1}{1 + \ln t}}x(t) = 0. \\
\rule{0pt}{15pt}
 & x(1)=1; \ \dot{x}(1)=0. 
 \end{array}\right.
\end{equation}

This system  falls within the scope of the "\textit{slow vanishing case}" $\int_{t_0}^{+\infty} \frac{\epsilon (t)}{t} dt =+ \infty$. In contrast to the cases examined later, we are not able to obtain an explicit form of the solution. We can only compute the solution by approximate numerical methods. The following table summarizes the results, it shows the convergence to the minimal norm solution, namely the zero element, and enlights the role played by the value of the coefficient $\alpha$. We will confirm this result in the following cases, where explicit solutions can be calculated,  and discuss it at the end of the section.

 \begin{equation*}
\begin{array}{|c||c||c||c||c|cc|ccr|} 
\hline 
\alpha & 1 & 2 & 3 &  4 \\
\hline 
x(10) & 0.319 & 0.038  & 0.04& -0. 06 \\ 
\hline
x(100) & -0.138 & -0.008& 0.001 & 6 \times  10^{-4}  \\
\hline
x(1000) & 0.048 & 0.002& 6 \times 10^{-5}& -2.7 \times 10^{-6}\\
\hline
\end{array}
\end{equation*}

\medskip

2. \textit{Case $\epsilon (t) = \frac{1}{t}$.}
 The system writes
\begin{equation}\label{algo7bb}
\left\{
\begin{array}{rcl}
 & \ \ddot{x}(t) + \displaystyle{\frac{\alpha}{t}} \dot{x}(t) + \displaystyle{\frac{1}{t}}x(t) = 0. \\
\rule{0pt}{15pt}
 & x(1)=1; \ \dot{x}(1)=0. 
 \end{array}\right.
\end{equation}

Despite the fact that this system does not fall within the scope of the "\textit{slow vanishing case}" $\int_{t_0}^{+\infty} \frac{\epsilon (t)}{t} dt =+ \infty$, it is not too far from it, and we will see that it has quite similar convergence properties. The following table summarizes the results and shows the role played by the value of the coefficient $\alpha$. 
The results are expressed in terms of $J_\gamma$ and $Y_\gamma$ which are the Bessel functions of the first and the second kind, respectively, with parameter $\gamma$. We use that $|J_\gamma(t)|=\mathcal O(t^{-\frac{1}{2}})$ and $|Y_\gamma(t)|=\mathcal O(t^{-\frac{1}{2}})$ (see \cite[Section 5.11]{Lebedev}).
\begin{itemize}
\item 
 For  $\alpha =1$ we get
$$
x(t)=\frac{J_1 (2)Y_0 (2 \sqrt{t}) - Y_1 (2) J_0 (2 \sqrt{t}) }{J_1 (2)Y_0 (2)  - J_0 (2)Y_1 (2) }.
$$
which gives 
$$|x(t)|=\mathcal O(\frac{1}{t^{\frac{1}{4}}}).$$ 

\item For  $\alpha =2$ we get
$$
x(t)=\frac{\left( Y_0 (2)-Y_1 (2) -Y_2 (2)\right) J_1 (2 \sqrt{t}) +\left( - J_0 (2) + J_1 (2) + J_2 (2)        \right) Y_1 (2 \sqrt{t})  }{\sqrt{t}  \left[ (J_2 (2)       - J_0 (2)) Y_1 (2) +   J_1 (2) ( Y_0 (2) - Y_2 (2))    
\right]  }
$$
which gives 
$$|x(t)|=\mathcal O(\frac{1}{t^{\frac{3}{4}}}).$$ 

\item 
 For  $\alpha =3$ we get
$$
x(t)=\frac{\left( Y_1 (2) -2 Y_2 (2) -Y_3 (2)                         \right) J_2 (2 \sqrt{t}) + \left(-J_1 (2) +2 J_2 (2) + J_3 (2) \right) Y_2 (2 \sqrt{t})}{t (J_3 (2) - J_1 (2)) Y_2 (2) + J_2 (2) ( Y_1 (2) - Y_3 (2)) }
$$
which gives 
$$|x(t)|=\mathcal O(\frac{1}{t^{\frac{5}{4}}}).$$ 
We observe that in each case the trajectory converges to zero, the minimal norm solution. The following table summarizes the rate of convergence of the trajectories to zero.
\end{itemize}
 \begin{equation*}
\begin{array}{|c||c||c||c||c|cc|ccr|} 
\hline 
\alpha & 1 & 2 & 3 &  4 \\
\hline
|x(t)| & \mathcal O(\frac{1}{t^{\frac{1}{4}}}) & \mathcal O(\frac{1}{t^{\frac{3}{4}}})& \mathcal O(\frac{1}{t^{\frac{5}{4}}})  &  \mathcal O (\frac{1}{t^{\frac{7}{4}}})   \\
\hline
\end{array}
\end{equation*}

\bigskip

This suggests that these results obey a simple rule. Indeed, one can show that for any $\alpha >0$
$$
|x(t)| = \mathcal O(\frac{1}{t^{\frac{2 \alpha -1}{4}}}),
$$
which is in accordance with the above table.

\medskip

A similar calculation  can be made with 
 $\epsilon (t)= \frac{1}{t^{\frac{1}{p}}}$    , $p\in \mathbb N*$. For example, with $\epsilon (t) = \frac{1}{\sqrt{t}}$ and $\alpha =4$, we obtain
$$|x(t)|=\mathcal O(\frac{1}{t^{\frac{15}{8}}}),$$ 
which is a faster convergence property  than in the previous case.
\bigskip

3. \textit{Case $\epsilon (t) = \frac{1}{t^2}$.}
 The system writes
\begin{equation}\label{algo7b}
\left\{
\begin{array}{rcl}
 & \ddot{x}(t) + \displaystyle{\frac{\alpha}{t}} \dot{x}(t) + \displaystyle{\frac{1}{t^2}}x(t) = 0. \\
\rule{0pt}{20pt}
 &x(1)=1; \ \dot{x}(1)=0. 
 \end{array}\right.
\end{equation}
As in the previous case, this example is outside  situations that were discussed in the paper.
The following table summarizes the results and shows the role played by the value of the coefficient $\alpha$. 

 \begin{equation*}
\begin{array}{|c||c||c||c||c|cc|ccr|} 
\hline 
\alpha & 1 & 2 & 3 &  4 \\
\hline 
x(t) & \cos (\ln t) & \frac{1}{\sqrt{x}}\left(\sqrt{3}\sin\left(\frac{1}{2} \sqrt{3} \ln t  \right) +3 \cos  \left(\frac{1}{2} \sqrt{3} \ln t  \right)\right)  & \frac{\ln t +1}{t} & \frac{1}{10}t^{\frac{1}{2}(-3 - \sqrt{5})}\left( (5+ 3 \sqrt{5})t^{\sqrt{5}} -3\sqrt{5} +5\right) \\ 
\hline
|x(t)| & not \ conv. & \mathcal O (\frac{1}{\sqrt{t}} )& \mathcal O \left( \frac{\ln t}{t}\right)  & \mathcal O \left( \frac{1}{t^{\frac{3-\sqrt{5}}{2}}} \right)   \\
\hline
\end{array}
\end{equation*}

\bigskip

4. \textit{Case $\epsilon (t) = \frac{1}{t^3}$.}
 The system writes
\begin{equation}\label{algo7bc}
\left\{
\begin{array}{rcl}
 & \ddot{x}(t) + \displaystyle{\frac{\alpha}{t}} \dot{x}(t) + \displaystyle{\frac{1}{t^3}}x(t) = 0. \\
\rule{0pt}{20pt}
 &x(1)=1; \ \dot{x}(1)=0. 
 \end{array}\right.
\end{equation}
This example falls within the scope of the 
"\textit{fast vanishing case}"  $\int_{t_0}^{+\infty} t \epsilon (t) dt < + \infty$. For $\alpha >3$, we know from Theorem \ref{tikh-fast} that the solution trajectory of 
(\ref{algo7bc}) converges to a minimizer of $\Phi$, which here can be any real number.
The following table summarizes the results. It confirms that the limit exists, but
    is different from the minimum norm solution, i.e., there is no asymptotic effect of  the Tikhonov regularizating term. The results were obtained using the following explicit form of the solution of 
(\ref{algo7bc})
$$
x(t) = C\left( \frac{1}{t}\right)^\frac{3}{2} \left( \left(Y_2(2) + 3Y_3 (2) -Y_4 (2)  \right)J_3 \left( 2 \sqrt{\frac{1}{t}}\right)  
+ \left(-J_2(2) - 3J_3 (2) +J_4 (2)  \right)Y_3 \left( 2 \sqrt{\frac{1}{t}}\right)\right) 
$$
$$
C= \left( \left( J_4 (2) -J_2 (2) \right) Y_3 (2) +         J_3 (2) \left( Y_2(2) - Y_4 (2)     \right)            \right) ^{-1}.
$$
 \begin{equation*}
\begin{array}{|c||c||c||c||c|cc|ccr|} 
\hline 
t & 10 & 100 & 1000 &  10000 \\
\hline 
x(t) & 0.74257 & 0.709214 & 0.70602 & 0.705703 \\ 
\hline
\end{array}
\end{equation*}

\bigskip

\textit{Comments:} \
The  results obtained in the case case $\epsilon (t) = \frac{1}{1+ \ln t}$ and  $\epsilon (t) = \frac{1}{t}$ are conform to \cite{ACPR} where, for   $\mbox{(AVD)}_{\alpha} $ system, it is shown that for a strongly convex potential function $\Phi$, the rate of convergence to the unique minimizer can be made arbitrarily fast with $\alpha$ large. In the cases above, we are asymptotically close to this situation, which may explain that a similar phenomenom occurs. This is an interesting question to be addressed for future research. 

However, in the case $\epsilon (t) = \frac{1}{t^2}$, and  more in the case $\epsilon (t) = \frac{1}{t^3}$, the asymptotic effect of the Tikhonov regularizing term is less effective. We are far from the critical case that ensures convergence to the minimum norm solution.

Another natural question concerns the comparison with the results of \cite{AttCza1} concerning the Tikhonov regularization of the heavy ball with friction method
$\mbox{(HBF)}_{\epsilon}$. In this case, the damping coefficient is a fixed positive number, and the trajectories converge to the minimum norm solution under the sole assumption $\int _{t_0}^{+\infty} \epsilon (t) dt = + \infty$.
Let us give an example where we compare the solutions of the two systems. A systematic study of this question is an interesting subject for further research.

\begin{equation}\label{algo7bd}
\mbox{(AVD)}_{\alpha, \epsilon} \left\{
\begin{array}{rcl}
 & \ddot{x}(t) + \displaystyle{\frac{3}{t}} \dot{x}(t) + \displaystyle{\frac{1}{t}}x(t) = 0. \\
\rule{0pt}{20pt}
 &x(1)=1; \ \dot{x}(1)=0. 
 \end{array}\right.
\end{equation}
and 
\begin{equation}\label{algo7be}
\mbox{(HBF)}_{\epsilon} \left\{
\begin{array}{rcl}
 & \ddot{y}(t) + 3 \dot{y}(t) + \displaystyle{\frac{1}{t}}y(t) = 0. \\
\rule{0pt}{20pt}
 &y(1)=1; \ \dot{y}(1)=0. 
 \end{array}\right.
\end{equation}

\bigskip

 \begin{equation*}
\begin{array}{|c||c||c||c||c|cc|ccr|} 
\hline 
t & 10 & 20 & 50 &  100 \\
\hline 
x(t) & -0.098 & 0.018 &- 0.010 & 0.006 \\ 
\hline
y(t) & 0.455 & 0.358 & 0.263 & 0.208\\
\hline
\end{array}
\end{equation*}

\bigskip

In this example,  we can see that  $\mbox{(AVD)}_{\alpha, \epsilon}$  
 outperforms the Tikhonov regularization of the heavy ball method. Indeed, a too large damping coefficient makes the latter system similar to the steepest descent method, and hence makes it relatively slow.
Specifically, the idea behind the $\mbox{(AVD)}_{\alpha, \epsilon}$  and the accelerated gradient method of Nesterov is to take a damping coefficient not too big,   to enhance the inertial effect.
This is a good strategy for the  values of $t$ that are not too large. Combining with a restart method provides effective numerical  method, that would be interesting to study for the $\mbox{(AVD)}_{\alpha, \epsilon}$ system.

\section{Conclusions}

Within the framework  of convex optimization, we  presented a second-order differential system $\mbox{(AVD)}_{\alpha, \epsilon}$  whose asymptotic behavior combines two distinct effects:

\begin{enumerate}
\item The asymptotically vanishing viscosity coefficient $\frac{\alpha}{t}$ corresponds to a continuous version of the  accelerated gradient method of Nesterov. It is associated with a rapid minimization property, $\Phi(x(t))-  \min_{\mathcal H}\Phi \leq \frac{C}{t^2}$. 

\smallskip

\item The Tikhonov regularization with asymptotically vanishing  coefficient $\epsilon (t)$ gives an asymptotic hierarchical minimization. In our context, it tends to make the trajectories of $\mbox{(AVD)}_{\alpha, \epsilon}$  converge strongly to  the minimizer of minimum norm.

\end{enumerate}

These two properties are important in optimization, which justifies our interest to combine them into a single dynamic system.
However, our analysis shows that they are in some way antagonistic. We obtained the above properties by requiring  the Tikhonov parametrization $t \mapsto \epsilon (t)$ to verify the following asymptotic behavior:

\begin{itemize}
 \item 
 Property (1) holds true when  $t \mapsto \epsilon (t)$ satisfies
$\int_{t_{0}}^{+\infty} t\epsilon (t)dt < +\infty$, which reflects a "fast vanishing" property of $\epsilon (t)$ as $t \to + \infty$.

 \item 
Property (2) holds true when  $t \mapsto \epsilon (t)$ satisfies
$\int_{t_0}^{+\infty} \frac{\epsilon (t)}{t} dt = + \infty$, which reflects a "slow vanishing" property of $\epsilon (t)$ as $t \to + \infty$.

\end{itemize} 

It is an open question  whether one can simultaneously obtain both above asymptotic properties in a simple dynamic system:  fast minimization, and strong convergence to the solution with minimum norm.
One may think using another  type of damped inertial dynamic system, for example involving a geometric damping as in \cite{aabr}, see also \cite{JM}. 
The numerical experiments also suggest that, in accordance with \cite{ACPR}, taking $\alpha$ large may improve the rate of convergence to the solution with minimum norm. Restarting may be also an efficient strategy, see \cite{Candes}, \cite{SBC}.

Our study of  $\mbox{(AVD)}_{\alpha, \epsilon}$ offers a new perspective on this subject and its applications. The study of the  algorithmic version of $\mbox{(AVD)}_{\alpha, \epsilon}$ is an interesting subject for further research.

\section{Appendix: Some auxiliary results}

In this section, we present some auxiliary lemmas that are used in the paper. These results can be found in \cite{ACPR}. We reproduce them here for the convenience of the reader.

To establish the weak convergence of the solutions of \eqref{edo01}, we will use Opial's Lemma \cite{Op}, that we recall in its continuous form. This argument was first used in \cite{Bruck} to establish the convergence of nonlinear contraction semigroups.

\begin{lemma}\label{Opial} Let $S$ be a nonempty subset of $\mathcal H$, and let $x:[0,+\infty[ \to \mathcal H$. Assume that 
\begin{itemize}
\item [(i)] for every $z\in S$, $\lim_{t\to\infty}\|x(t)-z\|$ exists;
\item [(ii)] every  sequential weak cluster point of $x(t)$, as $t\to\infty$, belongs to $S$.
\end{itemize}
Then $x(t)$ converges weakly as $t\to\infty$ to a point in $S$.
\end{lemma}

The following allows us to establish the existence of a limit for a real-valued function, as $t\to+\infty$:

\begin{lemma}\label{basic-edo} 
Let $\delta >0$, and let $w: [\delta, +\infty[ \rightarrow \mathbb R$ be a continuously differentiable function which is bounded from below. Assume
\begin{equation}\label{basic-edo1}
t\ddot{w}(t) + \alpha \dot w(t) \leq g(t),
\end{equation}
for some $\alpha > 1$, almost every $t>\delta$, and some nonnegative function $g\in L^1 (\delta, +\infty)$. Then, the positive part $[\dot w]_+$ of $\dot w$ belongs to $L^1(t_0,+\infty)$, and $\lim_{t\to+\infty}w(t)$ exists.
\end{lemma}
 
\begin{proof} 
Multiply \eqref{basic-edo1} by $t^{\alpha-1}$ to obtain
$$
\frac{d}{dt} \big(t^{\alpha} \dot w(t)\big) \leq t^{\alpha-1} g(t).
$$
By integration, we obtain
$$
\dot w(t) \leq \frac{{\delta}^{\alpha}|\dot w(\delta)|}{t^{\alpha} }  +  \frac{1}{t^{\alpha} }  \int_{\delta}^t  s^{\alpha-1} g(s)ds.
$$
Hence, 
$$
[\dot w]_{+}(t)  \leq \frac{{\delta}^{\alpha}|\dot w(\delta)|}{t^{\alpha} }  +  \frac{1}{t^{\alpha} }  \int_{\delta}^t  s^{\alpha-1} g(s)ds,
$$
and so,
$$
\int_{\delta}^{\infty} [\dot w]_{+}(t)  dt \leq \frac{{\delta}^{\alpha}|\dot w(\delta)|}{(\alpha -1) \delta^{\alpha -1}} +
  \int_{\delta}^{\infty}\frac{1}{t^{\alpha}}  \left(  \int_{\delta}^t  s^{\alpha-1} g(s) ds\right)  dt.
$$
Applying Fubini's Theorem, we deduce that
$$
\int_{\delta}^{\infty}\frac{1}{t^{\alpha}}  \left(  \int_{\delta}^t  s^{\alpha-1} g(s) ds\right)  dt =     
\int_{\delta}^{\infty}  \left(  \int_{s}^{\infty}  \frac{1}{t^{\alpha}} dt\right) s^{\alpha-1} g(s) ds 
= \frac{1}{\alpha -1} \int_{\delta}^{\infty}g(s) ds.
$$
As a consequence,
$$
 \int_{\delta}^{\infty} [\dot w]_{+}(t) dt \leq \frac{{\delta}^{\alpha}|\dot w(\delta)  |}{(\alpha -1) \delta^{\alpha -1}}  +
 \frac{1}{\alpha -1} \int_{\delta}^{\infty}g(s)  ds < + \infty.
$$
Finally, the function $\theta:[\delta,+\infty)\to\R$, defined by
$$\theta(t)=w(t)-\int_{\delta}^{t}[\dot w]_+(\tau)\,d\tau,$$
is nonincreasing and bounded from below. It follows that
$$\lim_{t\to+\infty}w(t)=\lim_{t\to+\infty}\theta(t)+\int_{\delta}^{+\infty}[\dot w]_+(\tau)\,d\tau$$ 
exists.
\end{proof}

The following is a continuous version of Kronecker's Theorem for series.

\begin{lemma}\label{basic-int} 
Take $\delta >0$, and let $f \in L^1 (\delta , +\infty)$ be nonnegative and continuous. Consider a nondecreasing function $\psi:[\delta,+\infty[\to ]0,+\infty[$ such that $\lim\limits_{t\to+\infty}\psi(t)=+\infty$. Then, 
$$\lim_{t \rightarrow + \infty} \frac{1}{\psi(t)} \int_{\delta}^t \psi(s)f(s)ds =0.$$ 
\end{lemma}
 \begin{proof} 
Given $\epsilon >0$, fix $t_\epsilon$ sufficiently large so that 
$$\int_{t_\epsilon}^{\infty}  f(s) ds \leq \epsilon.$$
Then, for $t \geq t_\epsilon$, split the integral $\int_{\delta}^t \psi(s)f(s) ds$ into two parts to obtain
$$\frac{1}{\psi(t)} \int_{\delta}^t \psi(s)f(s)ds =
\frac{1}{\psi(t)}\int_{\delta}^{t_\epsilon} \psi(s) f(s) ds
+ \frac{1}{\psi(t)}\int_{t_\epsilon}^t \psi(s) f(s) ds 
\leq   \frac{1}{\psi(t)}\int_{\delta}^{t_\epsilon} \psi(s)f(s) ds
+ \int_{t_\epsilon}^t  f(s) ds.$$
Now let $t\to+\infty$ to deduce that
$$0\le\limsup_{t\to+\infty}\frac{1}{\psi(t)}\int_{\delta}^t \psi(s)f(s)ds \le \epsilon.$$
Since this is true for any $\epsilon>0$, the result follows.
\end{proof}

\end{document}